\documentclass{amsart}
 \pdfoutput=1

\usepackage{graphicx}

\usepackage{subfig}
\usepackage{amsmath}
\usepackage{mathrsfs} 
\usepackage{amsfonts}
\usepackage{amssymb}%
\usepackage{amsthm}
\usepackage{amscd}
\usepackage{verbatim}
 \usepackage{mathabx}
 \usepackage{dsfont}
\usepackage{endnotes}
\usepackage{enumerate}
\usepackage{mathtools}
 \usepackage{comment}
\usepackage{tikz}
\usepackage{color}
\usepackage{algorithm}
\usepackage[noend]{algpseudocode}

\newcommand{\C}{\mathbb{C}}
\newcommand{\R}{\mathbb{R}}
\newcommand{\N}{\mathbb{N}}

\newcommand{\tn}{\textnormal}

\theoremstyle{plain}
\newtheorem{theorem}{Theorem}[section]
\newtheorem{corollary}[theorem]{Corollary}
\newtheorem{proposition}[theorem]{Proposition}
\newtheorem{lemma}[theorem]{Lemma}

\theoremstyle{definition}
\newtheorem{defn}[theorem]{Definition}
\theoremstyle{definition}

\theoremstyle{definition}

\theoremstyle{definition}
\theoremstyle{definition}
\newtheorem{question}{Question}
 \usepackage{hyperref}

\title{Enumerative Aspects of Nullstellensatz Certificates}
\author{Bart Sevenster${}^\dagger$, Jacob Turner${}^\dagger$}
\thanks{${}^\dagger$ Korteweg-de Vries Institute for Mathematics, University of Amsterdam, 1098 XG Amsterdam, Netherlands.}

\begin{document}
\begin{abstract}
 Using polynomial equations to model combinatorial problems has been a popular tool both in computational combinatorics as well as an approach to proving new theorems. In this paper, we look at several combinatorics problems modeled by systems of polynomial equations satisfying special properties. If the equations are infeasible, Hilbert's Nullstellensatz gives a certificate of this fact. These certificates have been studied and exhibit combinatorial meaning. In this paper, we generalize some known results and show that the Nullstellensatz certificate can be viewed as enumerating combinatorial structures. As such, Gr\"obner basis algorithms for solving these decision problems may implicitly be solving the enumeration problem as well.
\end{abstract}
\maketitle
\section{Introduction}

Polynomials and combinatorics have a long common history. Early in both the theories of graphs and of matroids, important polynomial invariants were discovered including the chromatic polynomial, the Ising model partition function, and the flow polynomial, many of which are generalized by the Tutte polynomial \cite{tutte1954contribution,white1987combinatorial,brylawski1992tutte,oxley1992matroid}. The notion of using generating functions is ubiquitous in many areas of combinatorics and many these polynomials can be viewed as such functions. Another general class of polynomials associated to graphs is the partition function of an edge coloring model \cite{delaharpe1993graph}.

On the other hand, polynomials show up not just as graph parameters. Noga Alon famously used polynomial equations to actually prove theorems about graphs using his celebrated ``Combinatorial Nullstellensatz'' \cite{alon1999combinatorial}. His approach often involved finding a set of polynomial equations whose solutions corresponded to some combinatorial property of interest and then studying this system. 

Modeling a decision problem of finding a combinatorial structure in a graph by asking if a certain system of polynomial equations has a common zero is appealing from a purely computational point of view. This allows these problems to be approached with well-known algebraic algorithms from simple Gaussian elimination to Gr\"obner basis algorithms \cite{de1995grobner,cox1992ideals}. Using polynomial systems to model these problems also seems to work nicely with semi-definite programming methods \cite{lasserre2002polynomials,laurent2005semidefinite,laurent2007semidefinite,parrilo2003semidefinite}.

This approach was used in \cite{Loera:2009:ECP:1552058.1552063}, where the question of infeasibility was considered. If a set of polynomial equations is infeasible, Hilbert's Nullstellensatz implies that there is a set of polynomials acting as a certificate for this infeasibility. For any given finite simple graph, a polynomial system whose solutions corresponded to independent sets of size $k$  was originally formulated by L\'aszl\'o Lov\'asz \cite{lovasz1994stable}, but other articles have studied the problem algebraically \cite{li1981independence,simis1994ideal}.

One of the interesting results in \cite{Loera:2009:ECP:1552058.1552063} was to show that for these particular systems of polynomials, the Nullstellensatz certificate contained a multivariate polynomial with a natural bijection between monomials and independent sets in the graph. As such, the Nullstellensatz certificate can be viewed as an enumeration of the independent sets of a graph. Furthermore, the independence polynomial can quickly be recovered from this certificate. 
Later, when modeling the set partition problem, a similar enumeration occurred in the Nullstellensatz certificate \cite{margulies2015complexity}.

This paper is directly inspired by these two results and we look at the different systems of polynomials given in \cite{Loera:2009:ECP:1552058.1552063} and show that this phenomenon of enumerative Nullstellensatz certificates shows up in all of their examples. We explain this ubiquity in terms of inversion in Artinian rings. One important example considered in \cite{Loera:2009:ECP:1552058.1552063} was $k$-colorable subgraphs. This problem has also been studied by analyzing polynomial systems in \cite{alon1992colorings,eliahou1992algebraic,matiyasevich2001some,de2015graph}. We generalize the polynomial systems used in \cite{Loera:2009:ECP:1552058.1552063} to arbitrary graph homomorphisms.

We also consider existence of planar subgraphs, cycles of given length, regular subgraphs, vertex covers, edge covers, and perfect matchings.  On the one hand, these results may be viewed negatively as they imply that a certificate for infeasibility contains much more information than necessary to settle a decision problem. There have also been papers on attempting efficient computations of Nullstellensatz certificates (\cite{de2008hilbert,de2011computing}) and we should expect that often this will be very hard. On the other hand, if one wishes to enumerate combinatorial structures, our results imply algorithms built from known algebraic techniques to solve this problem.

One polynomial system that does not fall into the general setting of the other examples is that of perfect matchings. While there is a natural system of equations modeling this problem that does have enumerative Nullstellensatz certificates, there is another system of polynomials generating the same ideal that does not. We spend the latter part of this paper investigating the second set and try to achieve some partial results in explaining what combinatorial information is contained in the Nullstellensatz certificates.

This paper is organized as follows. In Section \ref{sec1}, we review the necessary background on the Nullstellensatz and make our definitions precise. We then present our motivating example, independent sets, and explain how this particular problem serves as a prototype for other interesting problems. In Section \ref{sec2}, we prove a sufficient condition for a Nullstellensatz certificate to enumerate combinatorial structures and give several examples, some of them new and some reformulations of old ones, that satisfy this property. Lastly, in Section \ref{sec3}, we look at a system of polynomials whose solutions are perfect matchings that do not satisfy this sufficient condition and prove some results about the certificates.

\section{Background}\label{sec1}

Given a system of polynomials $f_1,\dots,f_s\in \C[x_1,\dots,x_n]$,  consider the set $$\mathcal{V}(f_1,\dots,f_s)=\{(\alpha_1,\dots,\alpha_n)\in\C^n|\;f_1(\alpha_1,\dots,\alpha_n)=\cdots=f_s(\alpha_1,\dots,\alpha_n)=0\}.$$ We call such a set an \emph{variety} (by an abuse of language, we use the term even for reducible and non-reduced sets in this paper). In particular, the empty set is a variety and if $\mathcal{V}(f_1,\dots,f_s)=\emptyset$, we say that the system of polynomials $f_1,\dots,f_s$ is \emph{infeasible}.

One version David Hilbert's famous Nullstellensatz states that a system $f_1,\dots,f_s $ $\in \C[x_1,\dots,x_n]$ is infeasible if and only if there exists polynomials $\beta_1,\dots,\beta_s\in \C[x_1,\dots,x_n]$ such that $\sum{\beta_if_i}=1$ (cf. \cite{cox1992ideals}). The set of polynomials $\beta_1,\dots,\beta_s$ are called a \emph{Nullstellensatz certificate} for the infeasibility of the system. The degree of the Nullstellensatz certificate is defined to be $\max\{\deg(\beta_1),\dots,\deg(\beta_s)\}$. We note that a Nullstellensatz certificate is dependent on the choice of polynomials defining the system. We will revisit this point later. The second observation is that Nullstellensatz certificates aren't unique. Often the Nullstellensatz certificates of greatest interest are those of minimum degree.

Research into an ``effective Nullstellensatz" has yielded general bounds for the degree of a Nullstellensatz certificate of a system of polynomials \cite{brownawell1987bounds,kollar1988sharp}. As such, the following  general algorithm for finding such certificates has been proposed (cf. \cite{de2008hilbert,de2011computing}). Suppose we have a system of $s$ equations in $n$ variables, $f_1,\dots,f_s$. We want to find $\beta_1,\dots,\beta_s$ such that $\sum_{i=1}^s{\beta_if_i}=1$. Let $\mathscr{M}_{n,k}$ denote the set of monomials of degree $\le k$ in $n$ variables.

\begin{algorithm} 
\caption{
Basic Outline of the NulLA Algorithm.}\label{thealgo}

\begin{algorithmic}
\State{Suppose we know from some theorem that a Nullstellensatz certificate must have degree at most $d$.}
\State{$i= 0$.}
 \While{$i\le d$}
 \Comment{test every degree for a certificate}
 \State{For $i\in [s]$, $\beta_i:=\sum_{M\in\mathscr{M}_{n,d}}{\alpha_M M}$.}
 \State{Let $\mathcal{L}$ be the empty set of linear equations.}
 \For{$M\in\mathscr{M}_{n,d}$}
 \State{Determine the coefficient of $M$ in $\sum{\beta_if_i}$, $L_M$.}
 \If{$M=1$}{ Append $L_M=1$ to $\mathcal{L}$.} 
 \Else{ Append $L_M=0$ to $\mathcal{L}$.}
 \EndIf
 \EndFor
 \State{Solve the system $\mathcal{L}$ if possible.}
 \If{$\mathcal{L}$ has a solution} { Output "Yes" and Exit While Loop.}
 \EndIf
 \If{ $i=d$}{ Output "No".}
 \Else{ $i\mapsto i+1$.}
 \EndIf
 \EndWhile
\end{algorithmic}
\end{algorithm}

We summarize the above pseudocode. First guess at the degree of the Nullstellensatz certificate and then consider generic polynomials $\beta_i$ in variables $x_1,\dots,x_n$ of said degree. Then the condition $\sum{\beta_if_i}=1$ can be reformulated as system of linear equations whose solutions give the coefficients each $\beta_i$ should have.

If the linear system has no solution, the guessed degree is increased. The general degree bounds guarantee that this process will terminate eventually, implicitly finding a valid certificate of minimal degree, provided the initial polynomial system was infeasible. This algorithm is similar to the XL style Gr\"obner basis algorithms studied in algebraic cryptography \cite{courtois2000efficient,ars2004comparison}. One way to understand the complexity of this algorithm is to look at the Nullstellensatz certificates that get produced for different systems of polynomials. This algorithm is one of the main motivations behind the inquiry into Nullstellensatz certificates.

In this paper, we will consider combinatorial problems modeled by systems of polynomials in $\C[x_1,\dots,x_n]$. The problems we consider will all come from the theory of finite graphs and all varieties will be zero dimensional.

\subsection{Motivating Example}\label{subsec:motex}

Let us give the first example of a polynomial system modeling a graph problem: independent sets. Lov\'asz gave the following set of polynomials for determining if a graph has an independent set of size $m$.

\begin{proposition}[\cite{lovasz1994stable}]\label{prop:indpolysystem}
 Given a graph $G$, every solution of the system of equations
 
 \begin{center}
 \begin{tabular}{r l}
 $x_i^2-x_i=0,$ & $i\in V(G),$\\
 $x_ix_j=0,$ &  $\{i,j\}\in E(G),$\\
 $\sum_{i=1}^n{x_i}=m$&
 \end{tabular}
 \end{center}
 corresponds to an independent set in $G$ of size $m$.
\end{proposition}

It is not hard to see that the equations in Proposition \ref{prop:indpolysystem} define a zero-dimensional variety whose points are in bijection with independent sets of size $m$. The first equation says that $x_i=0,1$ for all $i\in V(G)$. So every vertex is either in a set or not. The second equation says that two adjacent vertices cannot both be in the same set. The last equation says that precisely $m$ vertices are in the set.

If this system is infeasible, then the Nullstellensatz certificate has been explicitly worked out \cite{Loera:2009:ECP:1552058.1552063}. Many of properties of the certificate encodes combinatorial data. For example, one does not need to appeal to general effective Nullstellensatz bounds as the degree of the Nullstellensatz certificate can be taken to be the independent set number of the graph in question. Combining several theorems from that paper, the following is known.

\begin{theorem}[\cite{Loera:2009:ECP:1552058.1552063}]\label{thm:indsetcert}
 Suppose the system of equations in Proposition \ref{prop:indpolysystem} are infeasible. Then the minimum degree Nullstellensatz certificate is unique and has degree equal to $\alpha(G)$, the independence number of $G$. If 
 $$1=A(-m+\sum_{i=1}^n{x_i})+\sum_{\{i,j\}\in E(G)}{Q_{ij}x_ix_j}+\sum_{i=1}^n{P_i(x_i^2-x_i)},$$
 with certificate polynomials $A, Q_{ij}$, and $P_i$, then the degree of this certificate is realized by $A$; $\deg(P_i)\le \alpha(G)-1$ and $\deg(Q_{ij})\le \alpha(G)-2$. Furthermore, if the certificate is of minimum degree, the monomials in $A$ with non-zero coefficients can be taken to be precisely those of the form $\prod_{i\in I}{x_i}$ where $I$ is an independent set of $G$. Lastly, the polynomials $A$, $Q_i$, and $P_i$ all have positive real coefficients.
\end{theorem}

So Theorem \ref{thm:indsetcert} tells us for a minimum degree certificate, the polynomial $A$ enumerates all independent sets of $G$. The coefficients of the monomials, however, will not necessarily be one, so it is not precisely the generating function for independents sets. The precise coefficients were worked out in \cite{Loera:2009:ECP:1552058.1552063}. In the next section, we show that a similar theorem will hold for many other examples of zero-dimensional varieties coming from combinatorics.

\section{Rephrasing Nullstellensatz certificates as inverses in an Artinian ring}\label{sec2}

Recall that a ring $S$ is called \emph{Artinian} is it satisfies the descending chain condition on ideals, i.e. for every infinite chain of ideals $\cdots\subsetneq I_2\subsetneq I_1\subset S$, there is some $i$ such that $I_{k}=I_{k+1}$ for all $k\ge i$. Equivalently, viewed as a left module over itself, $S$ is finite dimensional, meaning it contains finitely many monomials.

Let $\mathcal{V}$ be a variety in $\C^n$ defined by the equations $f_1,\dots,f_s$. Although not standard, we do not assume that $\mathcal{V}$ is reduced or irreducible, which is to say that often the ideal $\langle f_1,\dots,f_s\rangle$ will not be radical or prime. An ideal $I(\mathcal{V}):=\langle f_1,\dots,f_s\rangle$ is radical if $g^\ell\in I(\mathcal{V})\implies g\in I(\mathcal{V})$. The quotient ring $\C[\mathcal{V}]:=\C[x_1,\dots,x_n]/I(\mathcal{V})$ is called the \emph{coordinate ring} of $\mathcal{V}$. It is an elementary fact from algebraic geometry that $\mathcal{V}$ is zero dimensional if and only if $\C[\mathcal{V}]$ is Artinian. This is true even for non-reduced varieties.

The polynomial systems coming from combinatorics are designed to have a solution if some combinatorial structure exists. In Subsection \ref{subsec:motex}, the combinatorial structure of interest was independent sets in a graph. Many of the polynomials systems that show up in examples have a particular set of equations that always play the same role. Given a polynomial system in $\C[x_1,\dots,x_n]$, we call a subset of the variables $x_i$ for $i\in I\subseteq[n]$ \emph{indicator variables} if the polynomial system includes the equations $x_i^2-x_i=0$ for $i\in I$ and $-m+\sum_{i\in I}{x_i}=0$ for some $m\in\N$. This was the case for the example in Subsection \ref{subsec:motex}.

The indicator variables often directly correspond combinatorial objects, e.g. edges or vertices in a graph. The equations $x_i^2-x_i=0$ means that object $i$ is either in some structure or not. Then the equation $-m+\sum_{i\in I}{x_i}=0$ says that there must be $m$ objects in the structure. The other equations in the polynomial system impose conditions the structure must satisfy.

Now suppose we are given an infeasible polynomial system $f_1=0,\dots,f_s=0$ in $R:=\C[y_1,\dots,y_n,x_1,\dots,x_p]$, where $y_1,\dots,y_n$ are indicator variables. Without loss of generality, let $f_1=-m+\sum_{i=1}^n{y_i}$ for some $m\in\N$. Then one way to find a Nullstellensatz certificate for this polynomial system is to find the inverse of $f_1$ in $R/\langle f_2,\dots,f_s\rangle$, which we denote $\beta_1$, and then express $f_1\beta_1-1$ as $\sum_{i=2}^s{\beta_if_i}$. The polynomials $\beta_1,\beta_2,\dots,\beta_s$ will be a Nullstellensatz certificate.

Throughout the rest of this section, we consider an infeasible polynomial system $f_1=0,\dots,f_s=0$ in $A:=\C[y_1,\dots,y_n,x_1,\dots,x_p]$ where $y_1,\dots,y_n$ are indicator variables and $f_1=-m+\sum_{i=1}^n{y_i}$ for some $m\in\N$ not equal to zero. We let $$R:=A/\langle f_2,\dots,f_s\rangle$$ and $\mathcal{V}=\tn{Spec}(R)$ be the variety defined by $f_2,\dots,f_s$. 

\begin{lemma}\label{lem:artin}
 The ring $R$ is Artinian if and only if for every $a\in\N$, the polynomial system $-a+\sum_{i=1}^n{y_i}=0,f_2=0,\dots,f_s=0$ has finitely many solutions.
\end{lemma}
\begin{proof}
We look at the variety defined by the equations $f_2,\dots,f_s$. Since $y_1,\dots,y_n$ are indicator variables, the equations $y_i^2-y_i=0$ are among the equations $f_2=0,\dots,f_s=0$. Thus any solution to these equations must have every $y_i$ is equal to either zero or one. Thus there are only finitely many $a\in\N$ such that  $-a+\sum_{i=1}^n{y_i}=0,f_2=0,\dots,f_s=0$ has a solution as $a$ must be less than or equal to $n$. Furthermore, each such polynomial system has finitely many solutions so the system $f_2=0,\dots,f_s=0$ only has finitely many solutions. Thus $\mathcal{V}$ is zero dimensional and the ring $R$ is Artinian. Conversely, if $R$ is Artinian, there can be only finitely many solutions to the system $f_2=0,\dots,f_s=0$ and thus only a finite subset of those solutions can also be a solution to the equation $-a+\sum_{i=1}^n{y_i}=0$.
\end{proof}

From here on out, we assume that $R$ is Artinian as this will be the case in every example we consider. This is the consequence of the fact that the polynomial systems are designed to have solutions corresponding to some finite combinatorial structure inside of some larger, yet still finite, combinatorial object. In our examples, we are always looking for some graph structure inside a finite graph. The examples we consider also satisfy another property that we shall assume throughout the rest of this section unless otherwise stated.

\begin{defn}
The system of polynomials $f_2=0,\dots,f_s=0$ in indicator variables $y_1,\dots,y_n$ is called \emph{subset closed} if
\begin{enumerate}[(a)]
 \item There is a solution of the system where $y_1=\dots=y_n=0$ and 
 \item Let $I\subseteq [n]$ and $\chi_I(i)=1$ if $i\in I$ and 0 else. If there is a solution of the system where $y_i=\chi_I(i)$, then for all $J\subseteq I$, there is a solution of the system where $y_j=\chi_J(j)$.
\end{enumerate}
\end{defn}

It is very easy to describe the monomials with non-zero coefficients appearing in the inverse of $f_1$ in the ring $R$. We look at the variety defined by the polynomials $f_2,\dots,f_s$ which consists of finitely many points. Suppose we are given a solution to the system $f_2=0,\dots,f_s=0$: $y_1=d_1,\dots,y_n=d_n$, and $x_1=a_1,\dots,x_p=a_p$. We can then map this solution to the point $(d_1,\dots,d_n)\in\{0,1\}^n$. We see that each solution to the system $f_2,\dots,f_s$ can be associated to a point on the $n$-dimensional hypercube $\{0,1\}^n$. Let $B$ be the subset of $\{0,1\}^n$ which is the image of this mapping.

Given $\textbf{d}=(d_1,\dots,d_n)\in \{0,1\}^n$, we can associate to it the monomial $y_{\textbf{d}}:=\prod_{i,d_i=1}{y_{i}}$. For any $\textbf{b}=(b_1\dots,b_n)\in B\setminus(0,\dots,0)$, the monomial $y_{\textbf{b}}$ regarded as a function restricted to $\mathcal{V}$ is not identically zero as there is a point in $\mathcal{V}$ where $y_{i}=1$ for all $b_i=1$ in $(b_1,\dots,b_n)$. 

\begin{lemma}\label{lem:onlygoodmons}
 If $f_1=0,\dots,f_s=0$ is subset closed, then for $\textbf{d}=(d_1,\dots,d_n)\notin B$, $y_{\textbf{d}}$ is in the ideal generated by $f_2,\dots,f_s$.
\end{lemma}
\begin{proof}
 If this were not the case, there would be a point $v=(c_1,\dots,c_n,\gamma_1,\dots,\gamma_p)\in\mathcal{V}$ such that $y_{\textbf{d}}(v)=1$. Let $I$ be the set $\{i\in[n]|c_i=1\}$. If $J=\{i\in[n]|d_i=1\}$, we see that $J\subseteq I$ and that there is a solution where $y_i=\chi_I(i)$. So there must be a solution where $y_i=\chi_J(i)$ by the property of being subset closed. This implies that $(d_1,\dots,d_n)\in B$ since $d_i=\chi_J(i)$, and so we have a contradiction.
 
 This means that for $\textbf{d}\notin B$, there is some $k\in\N$ such that $y_{\textbf{d}}^k=0$ in the ring $R$. The exponent $k$ may a priori be greater than one as we have not assumed that the ideal $\langle f_2,\dots,f_s\rangle$ is radical. However, we note that for any $\textbf{d}\in \{0,1\}^n$, $y_{\textbf{d}}^k=y_{\textbf{d}}$ for all $k\in\N$ because the equations $y_i^2-y_i$ for all $i\in [n]$ are among the polynomials $f_2,\dots,f_s$. Thus for $\textbf{d}\notin B$, $y_{\textbf{d}}=0$ in the ring $R$. 
\end{proof}

We can thus conclude that the monomials in the indicator variables in $R$ are in bijection with combinatorial structures satisfying the constraints of the polynomial system. In Proposition \ref{prop:indpolysystem}, the ring $\C[x_1,\dots,x_{|V(G)|}]$ modulo the first two sets of polynomials gives a ring whose monomials are of the form $\prod_{i\in I}{x_i}$, where $I$ indexes vertices in an independent set of $G$.

\begin{lemma}\label{lem:noposrel}
 Given $\textbf{b}_1,\dots,\textbf{b}_k\in B$, there are no polynomial relations of the form $\sum_{i=1}^k{a_iy_{\textbf{b}_i}}=0$ for $a_i\in\R_{\ge0}$ in $R$ except for all $a_i=0$. Similarly if all $a_i\in\R_{\le 0}$.
\end{lemma}
\begin{proof}
 We note that because of the polynomials $y_i^2-y_i=0$, any monomial $y_{\beta_i}$ can only take the value of zero or one when restricted to $\mathcal{V}$. The only set of non-negative real numbers whose sum is zero is the trivial case where all are zero. The proof for the second assertion is the same as the first.
\end{proof}

\begin{theorem}\label{thm:enumnull}
 There is a Nullstellensatz certificate $\beta_1,\dots,\beta_s$ for the system $f_1=0,\dots,f_s=0$ such that the non-zero monomials of $\beta_1$ are precisely the monomials $y_{\textbf{b}}$ for $\textbf{b}\in B$.
\end{theorem}
\begin{proof}
 We look at the Nullstellensatz certificate $\beta_1,\dots,\beta_n$ where $\beta_1f_1=1$ in $R$ so we can express $\beta_1f_1-1=\sum_{i=2}^s{\beta_if_i}$ in $A$. We now analyze what $\beta_1$ must look like. First of all we note that over $\C$, there is a power series expansion
 $$\frac{1}{-m+t}=-\frac{1}{m}\sum_{i\ge 0}{\bigg(\frac{t}{m}\bigg)^i}$$ viewed as a function in $t$. Replacing $t$ with $\sum_{i=1}^n{y_i}$, we get a power series in the indicator variables that includes every monomial in these variables with a negative coefficient.

 We consider the partial sums
 $$-\frac{1}{m}\sum_{i=0}^k{\bigg(\frac{t}{m}\bigg)^i},\qquad t=\sum_{i=1}^n{y_i}$$ first taken modulo the ideal generated by the polynomials of the form $y_i^2-y_i$. This gives a sum where the monomials are $y_{\textbf{d}}$ for $\textbf{d}\in\{0,1\}^n$ whose coefficients are all negative real numbers. We know from Lemma \ref{lem:onlygoodmons}, that the monomials in $\beta_1$ can be taken of the form $y_{\textbf{b}}$ for $\textbf{b}\in B$; the other monomials can be expressed in terms of $f_2,\dots,f_s$. So then those monomials that are equal to zero modulo $f_2,\dots,f_s$ can be removed, giving us another sum. If there is a relation of the form $\sum_{i}{a_i y_{\textbf{d}_i}}=\sum_{j}{c_j y_{\textbf{d}'_j}}$, we ignore it. Such relations simply mean that there is a non-unique way to represent this sum in $R$. Lastly, these partial sums converge to the inverse $\beta_1$ of $f_1$ in $R$ which is supported on the monomials of the form $y_{\textbf{b}}$ for $\textbf{b}\in B$, using the assumption that $R$ is 
Artinian.
\end{proof}

Theorem \ref{thm:enumnull} guarantees the existence of a Nullstellensatz certificate such that every possible combinatorial structure satisfying the constraints of $f_2,\dots,f_s$ is encoded in a monomial in the indicator variables appearing with non-zero coefficient in $\beta_1$. This is precisely the Nullstellensatz certificate found in Theorem \ref{thm:indsetcert} since any subset of an independent set is again an independent set. 

As it so happens, in Theorem \ref{thm:indsetcert} this Nullstellensatz certificate is also of minimal degree. However, it is not necessarily the case that the certificate given in Theorem \ref{thm:enumnull} is minimal. 

The most obvious way for minimality to fail is by reducing by the linear relations among the monomials $y_{\textbf{b}}$ for $\textbf{b}\in B$ with both negative and positive coefficients. However, if all such linear relations are homogeneous polynomials, we show these relations cannot reduce the degree of the Nullstellensatz certificate. 

\begin{defn}
 We say that the ideal $f_2,\dots,f_s$ has only \emph{homogeneous linear relations among the indicator variables} if every equation is of the form $\sum_{i=1}^\ell{a_iy_{\textbf{d}_i}}=0$ for $a_i\in\mathbb{R}$ and $\textbf{d}_i\in \{0,1\}^n$ and has the property that not all $a_i$ are positive or all negative and that $\deg(y_{\textbf{d}_i})=\deg(y_{\textbf{d}_j})$ for all $i,j\in[\ell]$.
\end{defn}
\begin{lemma}\label{lem:homlin}
 Let $\beta_1,\dots, \beta_s$ be a minimal Nullstellensatz certificate for the system $f_1=0,\dots, f_s=0$, and suppose that $\beta_1$ only has positive real coefficients. Then after adding homogeneous relations in the indicator variables to the system, there is a minimal degree Nullstellensatz certificate of the form $\beta_1,\beta'_2,\dots,\beta'_s$.
\end{lemma}
\begin{proof}
By adding homogeneous relations in the indicator variables, we claim it is impossible to reduce the degree of $\beta_1$ if it only has positive real coefficients. Let us try to remove the monomials of highest degree in $\beta_1$ by adding homogeneous linear relations in the indicator variables. We apply the first linear homogeneous relation to see which monomials we can remove. The relations are of the form $\sum_{i=1}^{k}{a_iy_{\textbf{b}_i}}=\sum_{j=k+1}^{\ell}{a_jy_{\textbf{b}_j}}$ where the $y_{\textbf{b}_k}$ are all monomials of a given degree $d$ and all $a_k\in \R_{\ge0}$. Thus we can potentially remove some of the monomials of degree $d$ in $\beta_1$ using such relations, but not all of them. This is true not matter how many linear homogeneous relations we apply; there will always be some monomials of highest degree remaining. However, we might be able to reduce the degree of the other $\beta_i$, $i\ge 2$ to get a Nullstellensatz certificate $\beta_1,\beta'_2,\dots,\beta'_s$ such that the degree of 
this new certificate is less than the degree of the original.
\end{proof}

Given the Nullstellensatz certificate $\beta_1,\dots,\beta_s$ guaranteed by Theorem \ref{thm:enumnull}, we note that $f_1\beta_1-1$ must consist entirely of monomials of the form $y_{\textbf{d}}$ for $\textbf{d}\notin B$. If $D=\max\{\deg(y_{\textbf{b}})|\;\textbf{b}\in B\}$, then $f_1\beta_1-1$ must be of degree $D+1$ as $f_1$ is linear. Let $M$ denote the set of monomials (in $A$) of $f_1\beta_1-1$.

We consider the following hypothetical situation where the Nullstellensatz certificate guaranteed by Theorem \ref{thm:enumnull} is not of minimal degree. Given a monomial $\mu\in M$, we have that $\mu=\sum_{i=2}^s{\mu_if_i}$ for some polynomials $\mu_i$, although this is not unique. We define $$\deg_R(\mu):=\max\{\deg(\mu_i),\tn{over all equalities } \mu=\sum_{i=2}^s{\mu_if_i}\}.$$ Then the degree of $\beta_1,\dots,\beta_s$ is $\max\{D, \deg_R(\mu)\tn{for } \mu\in M\}$. 

Suppose that the degree of the certificate is $\ge D+2$ and that for every $\max\{\deg_R(y_i\cdot\mu)|\mu\in M\}<\max\{\deg_R(\mu)|\mu\in M\}$ for every $i\in[n]$. Then define $\beta'_1:=\beta_1+m^{-1}(f_1\beta_1-1)$, which has degree $D+1$. 

Now we note that $$f_1\beta'_1-1=f_1\beta_1+m^{-1}f_1(f_1\beta_1-1)-1$$ $$=f_1\beta_1-f_1\beta_1+1+m^{-1}\sum_{i=1}^n{y_i(f_1\beta_1-1)}-1$$ $$=m^{-1}\sum_{i=1}^n{y_i(f_1\beta_1-1)}.$$ However, since every monomial in $m^{-1}\sum_{i=2}^n{y_i(f_1\beta_1-1)}$ is of the form $y_i\cdot\mu$ for some $\mu\in M$, we can express this polynomial as $\sum_{i=1}^s{\beta'_sf_s}$ where $\deg(\beta'_s)<\deg(\beta_s)$ since $\max\{\deg_R(y_i\cdot\mu)|\mu\in M\}<\max\{\deg_R(\mu)|\mu\in M\}$ for every $i\in[n]$. So we have found a Nullstellensatz certificate with smaller degree.

In the hypothetical situation above, we were able to drop the degree of the Nullstellensatz certificate by increasing the degree of $\beta_1$ by one. However, this construction can be iterated and it may be that the degree of $\beta_1$ must be increased several times before the minimal degree certificate is found. This depends on how high the degrees of $\beta_2,\dots,\beta_s$ are. It may also be the case that adding lower degree monomials also lowers the degree of the certificate.

\begin{lemma}\label{lem:maxbetaone}
 If $\beta_1,\dots,\beta_s$ be the Nullstellensatz certificate guaranteed by Theorem \ref{thm:enumnull} and $f_2,\dots,f_s$ have only linear homogeneous relations in the indicator variables. If $\max\{\deg(\beta_i),i\in[s]\}=\deg(\beta_i)$, then this is a Nullstellensatz certificate of minimal degree. 
\end{lemma}
\begin{proof}
 For any Nullstellensatz certificate $\beta'_1,\dots,\beta'_s$, we may assume without loss of generality that the monomials of $\beta_1$ form a subset of those in $\beta'_1$. Indeed we may use Lemma \ref{lem:onlygoodmons} to say that all monomials $y_{\textbf{b}}$ for $\textbf{b}\in B$ must appear in $\beta'_1$ unless there are linear homogeneous relations in the indicator variables. By Lemma \ref{lem:homlin}, reducing $\beta'_1$ alone by these relations will not reduce the degree of $\beta'_1,\dots,\beta'_s$. However, this implies that $\deg(\beta'_1)\ge\deg(\beta_1$ and thus that the degree of the Nullstellensatz certificate $\beta'_1,\dots,\beta'_s$ has degree $\ge\deg(\beta_1)$, which is the degree of the certificate $\beta_1,\dots,\beta_s$ by assumption. So $\beta_1,\dots,\beta_s$ is a Nullstellensatz certificate is of minimal degree.
\end{proof}

Lemma \ref{lem:maxbetaone} tells us that the Nullstellensatz certificate given in Theorem \ref{thm:enumnull} is na\"ively more likely to be of minimal degree when the degrees of $f_2,\dots,f_s$ are high with respect to the number of variables, implying that the degrees of $\beta_2,\dots,\beta_s$ are low.

\begin{proposition}\label{prop:lowdegredux}
 Let $f_1,\dots,f_s$ have only homogeneous linear relations in the indicator variables and $\beta_1,\dots,\beta_s$  be a Nullstellensatz certificate. Let $\beta_1$ be supported on the monomials $y_{\textbf{b}_1},\dots,y_{\textbf{b}_\ell}$ for $\textbf{b}_1,\dots,\textbf{b}_\ell\in B$. Then if for all $j\in[n]$ and $k\in[\ell]$, $y_jy_{\textbf{b}_k}=\sum_{i=2}^s{\alpha_{i}f_i}$ satisfies $\deg(\alpha_{i})\le\deg(y_{\textbf{b}_k})$ for all $i=2,\dots,s$, $\beta_1,\dots,\beta_s$ is a minimum degree Nullstellensatz certificate. 
 
 In addition, if $f_2=0,\dots,f_s=0$ is a polynomial system entirely in the indicator variables and there are only homogeneous linear relations, then there is a Nullstellensatz certificate of minimal degree of the form $\beta_1,\beta'_2,\dots,\beta'_s$, where $\beta_1$ is the coefficient polynomial guaranteed by Theorem \ref{thm:enumnull}.
\end{proposition}
\begin{proof}
  We know that $f_1\beta_1-1$ contains only monomials of the form $y_jy_{\textbf{b}_k}$ for $j\in[n]$ and $k\in[\ell]$. By assumption $y_jy_{\textbf{b}_k}=\sum_{i=2}^s{\alpha_{i}f_i}$ and satisfies $\deg(\alpha_{i})\le\deg(y_{\textbf{b}_k})$ for all $i=2,\dots,s$. This implies $\deg(\beta_1)\ge\deg(\beta_i)$ for all $i=2,\dots,s$. Then apply Lemma \ref{lem:maxbetaone}. 
  
  If we restrict our attention to a system $f_1=0,\dots,f_s=0$ only in the indicator variables, all homogeneous linear relations are in the indicator variables. Furthermore, any equation $y_jy_{\textbf{b}_k}=\sum_{i=2}^s{\alpha_{i}f_i}$ is a homogeneous linear relation in the indicator variables since these are the only variables in the equation. So these too can be ignored. We can then apply Lemma \ref{lem:homlin}.
\end{proof}

Proposition \ref{prop:lowdegredux} is a generalization of Theorem \ref{thm:indsetcert} as we can see from Proposition \ref{prop:indpolysystem} that all of the variables are indicator variables.

\subsection{Some examples with indicator variables}

We now reproduce the first theorem from \cite{Loera:2009:ECP:1552058.1552063}. This theorem establishes several polynomial systems for finding combinatorial properties of graphs and we shall see that all of them (except for one, which we have omitted from the theorem) satisfy the conditions of Theorem \ref{thm:enumnull}. Afterwards, we shall present a few new examples that also use indicator variables.

\begin{theorem}[\cite{Loera:2009:ECP:1552058.1552063}]\label{thm:examples}
$$ $$

\noindent 1. A simple graph $G=(V,E)$ with vertices numbered $1,\dots,n$ and edges numbered $1,\dots,e$ has a planar subgraph with $m$ edges if and only if the following system of equations has a solution:

\begin{center}
 \begin{tabular}{r l}
  $-m+\sum_{\{i,j\}\in E}{z_{ij}}=0$.\\
  $z_{ij}^2-z_{ij}=0$ & for all $\{i,j\}\in E$.\\
  $\prod_{s=1}^{n+e}{(x_{\{i\}k}-s)}=0$ & for $k=1,2,3$ and every $i\in [n]$.\\
  $\prod_{s=1}^{n+e}{(y_{\{ij\}k}-s)}=0$ & for $k=1,2,3$ and $\{i,j\}\in E$.\\
  $z_{\{ij\}}(y_{\{ij\}k}-x_{\{i\}k}-\Delta_{\{ij,i\}k})=0$ & for $i\in[n]$, $\{i,j\}\in E$, $k\in[3]$.\\
  $z_{\{uv\}}\prod_{k=1}^3{(y_{\{uv\}k}-x_{\{i\}k}-\Delta_{\{uv,i\}k})}=0$ & for $i\in[n]$, $\{u,v\}\in E$,~$u,v\ne i$.\\
  $z_{\{ij\}}z_{\{uv\}}\prod_{k=1}^3{(y_{\{ij\}k}-y_{\{uv\}k}-\Delta_{\{ij,uv\}k})}=0$ & for every $\{i,j\},\{u,v\}\in E$.\\
  $z_{\{ij\}}z_{\{uv\}}\prod_{k=1}^3{(y_{\{uv\}k}-y_{\{ij\}k}-\Delta_{\{uv,ij\}k})}=0$ & for every $\{i,j\},\{u,v\}\in E$.\\
  $\prod_{k=1}^3{(x_{\{i\}k}-x_{\{j\}k}-\Delta_{\{i,j\}k})}=0$ & for $i,j\in[n]$.\\
  $\prod_{k=1}^3{(x_{\{j\}k}-x_{\{i\}k}-\Delta_{\{j,i\}k})}=0$ & for $i,j\in[n]$.\\
  $\prod_{d=1}^{n+e-1}{(\Delta_{\{ij,uv\}k}-d)}=0$ & for $\{i,j\},\{u,v\}\in E$, $k\in[3]$.\\
  $\prod_{d=1}^{n+e-1}{(\Delta_{\{ij,i\}k}-d)}=0$ & for $\{i,j\}\in E$, $k\in[3]$.
 \end{tabular}
\end{center}

$$ $$
\noindent For $k=1,2,3$:
$$s_k\bigg(\prod_{\stackrel{i,j\in[n]}{i<j}}{(x_{\{i\}k}-x_{\{j\}k})}\prod_{\stackrel{i\in[n]}{\{u,v\}in E}}{(x_{\{i\}k}-y_{\{uv\}k})}\prod_{\stackrel{\{i,j\},}{\{u,v\}\in E}}{(y_{\{ij\}k}-y_{\{uv\}k})}\bigg)=1.$$

\noindent 2. A graph $G=(V,E)$ with vertices labeled $1,\dots,n$ has a $k$-colorable subgraph with $m$ edges if and only if the following systems of equations has a solution:

\begin{center}
 \begin{tabular}{r l}
  $-m+\sum_{\{i,j\}\in E}{y_{ij}}=0$\\
  $y_{ij}^2-y_{ij}=0$ & for $\{i,j\}\in E$.\\
  $x_i^k-1=0$ & for $i\in[n]$.\\
$y_{ij}(x_i^{k-1}+x_i^{k-2}x_j+\dots+x_{j}^{k-1})=0$ & for $\{i,j\}\in E$.
 \end{tabular}
\end{center}

\noindent 3. Let $G=(V,E)$ be a simple graph with maximum vertex degree $\Delta$ and vertices labeled $1,\dots,n$. Then $g$ has a subgraph with $m$ edges and edge-chromatic number $\Delta$ if and only if the following system of equations has a solution:

\begin{center}
 \begin{tabular}{r l}
  $-m+\sum_{\{i,j\}\in E}{y_{ij}}=0$\\
  $y_{ij}^2-y_{ij}=0$ & for $\{i,j\}\in E$.\\
  $y_{ij}(x_{ij}^\Delta-1)=0$ & for $\{i,j\}\in E$.
 \end{tabular}
\end{center}
$$s_i\prod_{\stackrel{j,k\in N(i)}{j<k}}{(x_{ij}-x_{ik})}=1\qquad\tn{for }i\in[n].$$
\end{theorem}

We can also look at the a system of polynomials that ask if there is a subgraph of graph homomorphic to another given graph. This is a generalization of Part 3 of Theorem \ref{thm:examples} as $k$-colorable subgraphs can be viewed as subgraphs homomorphic to the complete graph on $k$ vertices.

\begin{proposition}\label{prop:graphhom}
 Given two simple graphs $G$ (with vertices labeled $1,\dots,n$) and $H$, there is a subgraph of $G=(V,E)$ with $m$ edges homomorphic to $H$ if and only if the following system of equations has a solution:
 
 \begin{center}
  \begin{tabular}{r l}
   $-m+\sum_{\{i,j\}\in E}{y_{ij}}=0$\\
   $y_{ij}^2-y_{ij}=0$ & for all $\{i,j\}\in E$.\\
   $\bigg(\sum_{j\in N(i)}{y_{ij}}\bigg)\prod_{v\in V(H)}{(z_{i}-x_v)}=0$ & for all $i\in [n]$.\\
   $y_{ij}\prod_{\{v,w\}\in E(H)}{(z_i+z_j-x_v-x_w)}=0$ & for all $\{i,j\}\in E$.
  \end{tabular}
 \end{center}
\end{proposition}
\begin{proof}
 The variables $y_{ij}$ are the indicator variables which designate whether or not an edge of $G$ is included in the subgraph. The third set of equations says that if at least one of the edges incident to vertex $i$ is included in the subgraph, then vertex $i$ must map to a vertex $v\in V(H)$. The last set of equations says that if the edge $\{i,j\}$ is included in the subgraph, its endpoints must be mapped to the endpoints of an edge in $H$.
\end{proof}

We see that all of these systems of equations have indicator variables and that each system has only finitely many solutions. So Lemma \ref{lem:artin} says that the rings formed by taking a quotient by the ideal generated by all equations not of the form $-m+\sum_{i=1}^n{y_i}$ gives an Artinian ring. We also see that a subset of $k$-colorable subgraph is $k$-colorable, a subset of a planar subgraph is planar, and a subgraph of a $k$-edge colorable subgraph is $k$-edge colorable. Lastly, if a subgraph $F$ of $G$ is homomorphic to $H$, so is any subgraph of $F$ by restricting the homomorphism. So all of these systems are subset closed.

\begin{corollary}
 If the first system of equations in Theorem \ref{thm:examples} is infeasible, there is a Nullstellensatz certificate $\beta_1,\dots,\beta_s$ such that the monomials in $\beta_1$ are monomials in the variables $y_i$ in bijections with the planar subgraphs. If the second system in Theorem \ref{thm:examples} is infeasible, the same holds example that the monomials are in bijection with the $k$-colorable subgraphs. If the third system in Theorem \ref{thm:examples} is infeasible, the same holds except that the monomials are in bijection with the $k$-edge colorable subgraphs. Lastly, if the system of equations in Proposition \ref{prop:graphhom} is infeasible, the same holds except that the monomials are in bijection with the subgraphs homomorphic to $H$.
\end{corollary}
\begin{proof}
 This follows directly from Theorem \ref{thm:enumnull} and Theorem \ref{thm:examples}.
\end{proof}

None of the examples in Theorem \ref{thm:examples} satisfy the conditions of Proposition \ref{prop:lowdegredux} as the indicator variables are a proper subset of the variables in the system. The following theorem gives a few examples that only involve indicator variables. While only three of the four following examples satisfies the conditions of Proposition \ref{prop:lowdegredux}, we will see that Theorem \ref{thm:enumnull} can be useful in understanding minimum degree certificates if we can analyze the equations directly.

\begin{defn}
 Given a graph $G$, we say that a subgraph $H$ \emph{cages} a vertex $v\in V(G)$ if every edge incident to $v$ in $G$ is an edge in $H$.
\end{defn}

\begin{theorem}\label{thm:examples2}
 $$ $$
 1. A graph $G=(V,E)$ with vertices labeled $1,\dots,n$ has a regular spanning subgraph with $m$ edges if and only if the following system of equations has a solution:
 
 \begin{center}
  \begin{tabular}{r l}
   $-m+\sum_{\{i,j\}\in E}{y_{ij}}=0$\\
   $y_{ij}^2-y_{ij}=0$& for all $\{i,j\}\in E$.\\
   $\sum_{j\in N(i)}{y_{ij}}=\sum_{k\in N(\ell)}{y_{k\ell}}$ & for every $i,\ell\in[n]$.
  \end{tabular}
 \end{center}
Furthermore, if the system is infeasible, there is a minimal degree Nullstellensatz certificate of the form $\beta_1,\beta'_2,\dots,\beta'_s$, where $\beta_1$ is the coefficient polynomial guaranteed in Theorem \ref{thm:enumnull}.
$$ $$
\noindent 2. A graph $G=(V,E)$ with vertices labeled $1,\dots,n$ has a $k$-regular subgraph with $m$ edges if and only if the following system of equations has a solution:
 
 \begin{center}
  \begin{tabular}{r l}
   $-m+\sum_{\{i,j\}\in E}{y_{ij}}=0$\\
   $y_{ij}^2-y_{ij}=0$& for all $\{i,j\}\in E$.\\
   $(\sum_{j\in N(i)}{y_{ij}})(\sum_{j\in N(i)}{y_{ij}}-k)=0$ & for every $i\in[n]$.
  \end{tabular}
 \end{center}
Furthermore, if the system is infeasible, if there exists an edge in a maximum $k$-regular subgraph such that for both of its endpoints, there is an edge incident to it that is in no maximum $k$-regular subgraph, then there is a minimal degree Nullstellensatz certificate of the form $\beta_1,\beta'_2,\dots,\beta'_s$, where $\beta_1$ is the coefficient polynomial guaranteed in Theorem \ref{thm:enumnull}.
$$ $$
 \noindent 3. A graph $G=(V,E)$ with vertices labeled $1,\dots,n$ has a vertex cover of size $m$ if and only if the following system of equations has a solution:
 
 \begin{center}
  \begin{tabular}{r l}
   $-(n-m)+\sum_{i\in[n]}{y_i}=0$\\
   $y_i^2-y_i=0$ & for all $i\in[n]$.\\
   $y_iy_j=0$ & for all $\{i,j\}\in E$.
  \end{tabular}
 \end{center}
Furthermore, if the system is infeasible, there is a  Nullstellensatz certificate $\beta_1,\dots,\beta_s$ of minimal degree such the monomials in $\beta_1$  are in bijection with the independent sets of $G$.

  \noindent 4. A graph $G$ with vertices labeled $1,\dots,n$ and $e$ edges has an edge cover of size $m$ if and only if the following system of equations has a solution:
  
  \begin{center}
   \begin{tabular}{r l}
    $-(e-m)+\sum_{\{i,j\}\in E}{y_{ij}}=0$\\
    $y_{ij}^2-y_{ij}=0$ & for all $\{i,j\}\in E$.\\
    $\prod_{j\in N(i)}{y_{ij}}=0$ & for all $i\in[n]$.
   \end{tabular}
  \end{center}
  Furthermore, if the system is infeasible,  there is a minimal degree Nullstellensatz certificate $\beta_1,\dots,\beta_s$ such that the monomials of $\beta_1$ correspond to the subgraphs of $G$ that cage no vertex of $G$.
 
\end{theorem}

\begin{proof}
 
We first prove Part 1. First we show that a solution to the system imply the existence of a regular spanning subgraph of size $m$. The indicator variables correspond to edges that will either be in a subgraph satisfying the last set of equations or not. The last set of equations say that every pair of vertices must be incident to the same number of edges in the subgraph. The last equations are homogeneous linear equations and so we use Proposition \ref{prop:lowdegredux} to prove that Nullstellensatz certificate guaranteed in Theorem \ref{thm:enumnull} is a minimal degree certificate.

Now we move to Part 2. Once again, the indicator variables correspond to edges that will either be in the subgraph or not. The last set of equations say that the number of that every vertex must be incident to $k$ edges in the subgraph or 0 edges. The last equations are not homogeneous linear relations. Now suppose that there is an edge $\{i,j\}$ that is in a maximum $k$-regular subgraph and $\{i,\ell_1\}$ and $\{\ell_2,j\}$ are edges in none.

The polynomial $\beta_1$ in the certificate $\beta_1,\dots,\beta_s$ given by Theorem \ref{thm:enumnull} contains a monomial for every $k$-regular subgraph. At least one of these monomials contains the variable $y_{ij}$. There are only two linear relations in which $y_{ij}$ appears: $(\sum_{s\in N(i)}{y_{is}})(\sum_{s\in N(i)}{(y_{is}-k)})=0$ and $(\sum_{s\in N(i)}{y_{si}})(\sum_{s\in N(j)}{(y_{sj}-k)})=0$. The former equation involves the variable $y_{i\ell_1}$ and the latter the variable $y_{\ell_2j}$. But neither of these variables appear in monomials of maximal degree by assumption. Therefore monomials of maximal degree involving $y_{ij}$ cannot be gotten rid of by the polynomials $f_2,\dots,f_s$. So the total degree of $\beta_1$ cannot be reduced.
 
Now we prove Part 3. We first consider a different system modeling vertex cover:

\begin{center}
  \begin{tabular}{r l}
   $-m+\sum_{i\in[n]}{x_i}=0$\\
   $x_i^2-x_i=0$ & for all $i\in[n]$.\\
   $(x_i-1)(x_j-1)=0$ & for all $\{i,j\}\in E$.
  \end{tabular}
 \end{center}

 In this system, the indicator variables correspond to vertices that will be either in a vertex cover or not. The last set of equations say that for every edge, at least one of its endpoints must be included the in the vertex cover. However, this system is not subset closed, in fact it is the opposite. If a set is a vertex cover, so is any \emph{superset}. So for Theorem \ref{thm:enumnull} to be applicable, we make the variable change $-y_i=x_i-1$. Plugging this variable change in gives us the equations in the statement in the theorem and is now subset closed. However, it defines an isomorphic ideal. We then note that these equations model independent set on the same graph and  use Theorem \ref{thm:indsetcert}.

  Lastly, we prove Part 4. Like in Part 3, we first consider the following system:
  
  \begin{center}
   \begin{tabular}{r l}
    $-m+\sum_{\{i,j\}\in E}{x_{ij}}=0$\\
    $x_{ij}^2-x_{ij}=0$ & for all $\{i,j\}\in E$.\\
    $\prod_{j\in N(i)}{(x_{ij}-1)}=0$ & for all $i\in[n]$.
   \end{tabular}
  \end{center}
  
  In this system, the indicator variables correspond to edges that are in the edge cover or not. The last equations say that for every vertex, at least one of its incident edges must be in the edge cover. Once again, this system is the opposite of being subset closed: any superset of an edge cover is an edge cover. So once again we make a variable substitution, this time $-y_{ij}=x_{ij}-1$. Plugging in gives us the system in the statement of the theorem. We then use Proposition \ref{prop:lowdegredux}, noting there are no linear relations among the indicator variables, and note that those square free monomials that get sent to zero are those divisible by a monomial of the form $\prod_{j\in N(j)}{x_{ij}}$. If a monomial is not divisible by a monomial of such a form, it corresponds to a subgraph that cages no vertex.

\end{proof}

We see from Part 2 of Theorem \ref{thm:examples2} that whether or not a minimal degree Nullstellensatz certificate exists that enumerates all combinatorial structures satisfying the polynomial constraints is sensitive to the input data. We also from the proofs of Parts 3 and 4 how Theorem \ref{thm:enumnull} might not be applicable. However, in the case of a superset closed system, it is generally possible to change it to a subset closed system using the change of variables exhibited in the proof of Theorem \ref{thm:examples2}.

While the systems of equations for Parts 3 and 4 of Theorem \ref{thm:examples2} are not the most obvious ones, because they can be obtained from a more straightforward system by a linear change of basis, we have the following Corollary.

\begin{corollary}
$$ $$
 Part 1.  A graph $G=(V,E)$ with vertices labeled $1,\dots,n$ has a vertex cover of size $m$ if and only if the following system of equations has a solution:
 \begin{center}
  \begin{tabular}{r l}
   $-m+\sum_{i\in[n]}{x_i}=0$\\
   $x_i^2-x_i=0$ & for all $i\in[n]$.\\
   $(x_i-1)(x_j-1)=0$ & for all $\{i,j\}\in E$.
  \end{tabular}
 \end{center}
Furthermore, if the system is infeasible, the degree of a minimum degree Nullstellensatz certificate is the independence number of $G$.
$$ $$
Part 2.  A graph $G=(V,E)$ with vertices labeled $1,\dots,n$ and $e$ edges has an edge cover of size $m$ if and only if the following system of equations has a solution:
  
  \begin{center}
   \begin{tabular}{r l}
    $-m+\sum_{\{i,j\}\in E}{x_{ij}}=0$\\
    $x_{ij}^2-x_{ij}=0$ & for all $\{i,j\}\in E$.\\
    $\prod_{j\in N(i)}{(x_{ij}-1)}=0$ & for all $i\in[n]$.
   \end{tabular}
  \end{center}
 Furthermore, if the system is infeasible, the degree of a minimum degree Nullstellensatz certificate is equal to the maximum number of edges a subgraph of $G$ can have such that no vertex of $G$ is caged. 
\end{corollary}
\begin{proof}
 
For both parts, the correctness of the system of equations was proven in Theorem \ref{thm:examples2}. Furthermore, these systems are equivalent to the systems in Parts 3 and 4, respectively, in Theorem \ref{thm:examples2} after an invertible linear change of basis. This means that if $\beta_1,\dots,\beta_s$ is a minimum degree Nullstellensatz certificate for the systems in Theorem \ref{thm:examples2}, then applying an an invertible linear change of basis to the variables in the $\beta_i$ preserves their degrees. Thus the degrees any minimum degree Nullstellensatz certificate for the systems in the statement of the corollary must be the same.
 
\end{proof}

\section{Perfect Matchings}\label{sec3}

We now turn our attention to the problem of determining if a graph has a perfect matching via Nullstellensatz certificate methods. Unlike many of the problems considered in the previous section, this problem is not $\mathsf{NP}$-complete. Edmond's blossom algorithm determines if a graph $G=(V,E)$ has a perfect matching in time $O(|E||V|^{1/2})$. The following set of equations has been proposed for determining if a graph $G$ has a perfect matching.

\begin{equation}\label{eq:firstmatch}
 \begin{tabular}{r l}
  $\sum_{j\in N(i)}{x_{ij}}=1$& $i\in V(G)$\\
  $x_{ij}x_{jk}=0$ & $\forall i\in V(G)$, $j,k\in N(i)$
 \end{tabular}
\end{equation}
where $N(i)$ denotes the neighborhood of vertex $i$. The first equation says that a vertex must be incident to at least one edge in a perfect matching and the second equation says it can be incident to at most one edge. So indeed, these equations are infeasible if and only if $G$ does not have a perfect matching. However, there is not yet a complete understanding of the Nullstellensatz certificates if this system is infeasible \cite{marguliesberkley}.

We note that the equations $x_{ij}^2-x_{ij}$ can easily be seen to be in the ideal generated by the Equations \ref{eq:firstmatch}. Thus the variables $x_{ij}$ are indicator variables. However, there is no equation of the form $-m+\sum{x_{ij}}$, so we are not in the situation required to apply Theorem \ref{thm:enumnull}. That said, there still exists a Nullstellensatz certificate such that the non-zero monomials are precisely those corresponding to matchings in the graph $G$.

We observe that a matching on a graph $G$ corresponds precisely to an independent set of its line graph $L(G)$. In fact, there is a bijection between independent sets of $L(G)$ and matchings of $G$. This suggests a different set of equations for determining perfect matchings of $G$ that mimic those in Proposition \ref{prop:indpolysystem}.
\begin{equation}\label{eq:secondmatch}
 \begin{tabular}{r l}
   $x_{ij}^2-x_{ij}=0,$ & $\{i,j\}\in E(G),$\\
 $x_{ij}x_{ik}=0,$ &  $\{i,j\},\{i,k\}\in E(G)$,\\
 $\sum_{\{i,j\}\in E(G)}{x_{ij}}$ &$=|V(G)|/2$.
 \end{tabular}
\end{equation}

It quickly follows from Proposition \ref{prop:indpolysystem} that the solutions to this system forms a zero-dimensional variety whose solutions correspond to perfect matchings. However, we also know from Theorem \ref{thm:indsetcert} that if the system is infeasible then there is a unique minimum degree Nullstellensatz certificate whose degree is the size of a maximum matching of $G$. Furthermore, the coefficient polynomial for the equation $\sum{x_{ij}}=|V(G)|/2$ in this certificate has monomials precisely corresponding to matchings in $G$. 

Equations \ref{eq:firstmatch} and Equations \ref{eq:secondmatch} define the same variety as a set. We now want to find a way of turning a Nullstellensatz certificate for Equations \ref{eq:secondmatch} into a Nullstellensatz certificates for Equations \ref{eq:secondmatch}. This should be possible if Equations \ref{eq:firstmatch} and Equations \ref{eq:secondmatch} both define the same ideal. It is sufficient to show that both generate a radical ideal. We have the following lemma (cf. \cite{kreuzer2000computational}).

\begin{proposition}[Seidenberg's Lemma]\label{prop:seidenberg}
 Let $I\subset k[x_1,\dots,x_n]$ be a zero dimensional ideal. Suppose that for every $i\in[n]$, there is a non-zero polynomial $g_i\in I\cap k[x_i]$ such that $g_i$ has no repeated roots. Then $I$ is radical.
\end{proposition}

We see that the polynomials of the form $x_{ij}^2-x_{ij}$ satisfy the conditions of Proposition \ref{prop:seidenberg} and so both ideals are indeed radical. By Theorem \ref{thm:indsetcert}, if Equations \ref{eq:secondmatch} are infeasible, we have a Nullstellensatz certificate of the form 
$$1=A\bigg(-\frac{|V(G)|}{2}+\sum_{\{i,j\}\in E(G)}{x_{ij}}\bigg)+\sum_{\stackrel{\{i,j\}\ne \{i,k\}}{\in E(G)}}{Q^i_{jk}x_{ij}x_{ik}}+\sum_{\{i,j\}\in E(G)}{P_{ij}(x_{ij}^2-x_{ij})},$$
 where $A$ is a polynomial whose monomials are in bijection with matchings of $G$ and all coefficients are positive real numbers.
If Equations \ref{eq:firstmatch} are infeasible, we denote by the polynomials $\Delta_i$ and $\Theta^i_{jk}$ a Nullstellensatz certificate such that
$$1=\sum_{i\in V(G)}{\Delta_i\bigg[(\sum_{j\in N(i)}{x_{ij}})-1\bigg]}+\sum_{\{i,j\}\ne\{i,k\}\in E(G)}{\Theta^i_{jk}x_{ij}x_{ik}}.$$

\begin{proposition}\label{prop:matchingcert}
 If Equations \ref{eq:firstmatch} are infeasible, then there is a Nullstellensatz certificate $\Delta_i$, $i\in V(G)$, and $\Theta^i_{jk}$ for $\{i,j\}\ne\{j,k\}\in E(G)$ such that 
 \begin{enumerate}[(a)]
  \item The degree of each $\Delta_i$ is the size of a maximal matching of $G$.
  \item For every matching $M$ of $G$, the monomial $\prod_{\{i,j\}\in M}{x_{ij}}$ appears with non-zero coefficient in $\Delta_i$ for all $i\in V(G)$.
  \item The degree of $\Theta^i_{jk}$ is less than or equal to the degree of $\Delta_\ell$ for all $i,j,k,\ell\in V(G)$.
 \end{enumerate}
\end{proposition}
\begin{proof}	
 
First we note that
$$\frac{1}{2}\sum_{i\in V(G)}{\bigg[(\sum_{j\in N(i)}{x_{ij}})-1\bigg]}=-\frac{|V(G)|}{2}+\sum_{\{i,j\}\in E(G)}{x_{ij}}.$$
 Then for $\{i,j\}\in E(G)$ we have that
 $$P_{ij}(x_{ij}^2-x_{ij})=P_{ij}\bigg(x_{ij}[-1+\sum_{\{i,k\}\in E(G)}{x_{ik}}]-\sum_{\stackrel{\{i,k\}\in E(G)}{j\ne k}}{x_{ij}x_{ik}}\bigg).$$
 So if $A$, $P_{ij}$ and $Q^i_{jk}$ are a Nullstellensatz certificate for Equations \ref{eq:secondmatch}, then we see that if we set $$\Delta_i:=\frac{1}{2}A+\sum_{j\in N(i)}{P_{ij}x_{ij}}\qquad\tn{ and}$$ $$\Theta^i_{jk}:=Q^i_{jk}-P_{ij}\sum_{\stackrel{\{i,k\}\in E(G)}{j\ne k}}{x_{ij}x_{ik}},$$ that we get a Nullstellensatz certificate for Equations \ref{eq:firstmatch}. Since $\deg(P_{ij})<\deg(A)$ and both have only positive real coefficients, $\deg(\Delta_i)=\deg(A)$, which is the size of a maximal matching of $G$, using Theorem \ref{thm:indsetcert}. This also implies Part (b) of the statement. Lastly, we note that since $\deg(Q^i_{jk})\le \deg(A)-2$ that $\Theta^i_{jk}$ has degree at most $\deg(A)=\deg(\Delta_i)$, again using Theorem \ref{thm:indsetcert}.
\end{proof}

While Proposition \ref{prop:matchingcert} implies the existence of an enumerative Nullstellensatz certificate similar to that in Theorem \ref{thm:indsetcert}, it is not necessarily of minimal degree. In fact many times it will not be. Consider the following result.

\begin{theorem}\label{thm:bipcert}
A loopless graph $G$ has a degree zero Nullstellensatz certificate $\beta_1,\dots,\beta_s$ for Equations \ref{eq:firstmatch} if and only if $G$ is bipartite and the two color classes are of unequal size. Furthermore, we can choose such a Nullstellensatz certificate such that for each non-zero $\beta_i$, $|\beta_i|^{-1}$ can be take to be equal to the difference in size of the independent sets.
\end{theorem}

\begin{proof}
Let $G = (V,E)$ and $f_i = \sum_{j \in N(i)} x_{ij} -1$ for $i \in V$. Suppose the graph $G$ is bipartite and has two  color classes $A$ and $B$, such that $|A| > |B|$. Let $c = \frac{1}{|A| - |B|}$, then we have that
\[
\sum_{i \in A} c f_i + \sum_{j \in B} -c f_j= 1,
\]
so this gives a Nullstellensatz certificate of degree $0$ for $G$.

Conversely, suppose that $G$ has a degree zero Nullstellensatz certificate $\beta_1,...,\beta_s$. Clearly, the coefficients of the equations of the form $x_{ij}x_{jk}$ have to be zero. Now for some vertex $v_i$, let the equation $f_i $ have $\beta_i = c$, for $c \in \C$. Then for all $j \in N(i)$ we have that $\beta_j = -c$. Repeating this argument, we see that $G$ can not have any odd cycles. Furthermore, for the sum to be unequal to $0$, we need the sizes of the color classes to be unequal. 
\end{proof}

We see that as the size of the graphs being considered grows, the difference in the degree of Nullstellensatz certificate given in Proposition \ref{prop:matchingcert} and a the degree of a minimal degree certificate can grow arbitrarily large since Theorem \ref{thm:bipcert} gives an infinite family of graphs with degree zero Nullstellensatz certificates.

We analyze the time complexity of the NulLA algorithm if it is promised a connected bipartite graph with independent sets of unequal size for returning the result that the equations are infeasible. The algorithm first assumes that the polynomial equations has a Nullstellensatz certificate of degree zero, which we know from Theorem \ref{thm:bipcert} to be true in this case. Letting $f_i=\sum_{j\in N(i)}{x_{ij}}-1$ and $g^i_{jk}=x_{ij}x_{ik}$, then the algorithm will try to find constants $\alpha_i$ and $\beta^i_{jk}$ such that $\sum{\alpha_if_i}+\sum{\beta^i_{jk}x_{ij}x_{ik}}=1$. However, we immediately see that $\beta^i_{jk}=0$ for all $\{i,j\},\{i,k\}\in E(G)$. 

So we consider an augmented matrix $M|v$ with columns labeled by the constants $\alpha_i,\beta^i_{jk}$ and rows for each linear relation that will be implied among the constants, which we now determine. Each variable $x_{ij}$, $\{i,j\}\in E(G)$, appears as a linear term in exactly two polynomials: $f_i$ and $f_j$. We see that this imposes the relation $\alpha_i+\alpha_j=0$ for $\{i,j\}\in E(G)$. Because of the $-1$ appearing in each $f_i$, we also have that $\sum_{i\in [n]}{\alpha_i}=-1$. Lastly, since each $g^i_{jk}$ has no monomial in common with $g^{i'}_{j'k'}$, there are no relations among the $\beta^i_{jk}$. So we see that the number of rows of $M$ is $|E(G)|+1$.

The matrix $M$ can then be described as follows: If we restrict to the columns labeled by the $\alpha_i$, we get a copy of the incidence matrix of $G$ along with an extra row of all one's added to the bottom. The columns labeled by $\beta^i_{jk}$ are all zero columns. The augmented column $v$ has a zero in every entry except the last, which is contains a negative one.

The NulLA algorithm seeks to determine if this linear system has a solution. Since we have a matrix with $|V(G)|$ nontrivial columns and $|E(G)|+1\ge |V(G)|$ rows. This takes time $\Omega(|V(G)|^\omega)$ to run (where $\omega$ is some constant $> 2$, although conjectured to asymptotically approach 2, depending on the complexity of matrix multiplication \cite{coppersmith1982asymptotic}). However, two-coloring a graph and counting the size of the two independent sets can be done in time $O(|V(G)|+|E(G)|)=O(|V(G)|^2)$. So even in the best case scenario, the NulLA algorithm is not an optimal algorithm.

\subsection{Nullstellensatz certificates for Odd Cliques}

We now turn our attention to another question inspired by Proposition \ref{prop:matchingcert}. When is the Nullstellensatz certificate given in that theorem of minimal degree? Surprisingly, this turns out to be the case for odd cliques. This is especially unappealing from an algorithmic standpoint as any graph with an odd number of vertices clearly cannot have a perfect matching.

Throughout the rest of the is section, we take $G=K_n$, for $n$ odd. To prove our result, we will work over the ring $R=\C[x_{ij}|\;\{i,j\}\in [n], i\ne j]/I$ where $I$ is the ideal generated by the polynomials $x_{ij}^2-x_{ij}$ and $x_{ij}x_{ik}$ for $\{i,j\},\{i,k\}\in E(K_n)$. We will be doing linear algebra over this ring as it is the coefficient polynomials $\Delta_i$ that we are most interested in. Furthermore, we note that adding the equations $x_{ij}^2-x_{ij}$ does not increase the degree of the polynomials $\Delta_i$ in a certificate and it is convenient to ignore square terms.

Working over $R$, we now want to find polynomials $\Delta_i$, $i\in [n]$, such that $\sum_{i\in[n]}{\Delta_i(\sum_{j\in N(i)}{x_{ij}}-1)}=1$. Our goal is to prove that each $\Delta_i$ has degree $\lfloor n/2\rfloor$, which is the size of a maximum matching in $K_n$, for $n$ odd.

We already knew from Theorem \ref{thm:bipcert} that any Nullstellensatz certificate for $K_n$ must be of degree at least one. For the proof of the statement, it will be convenient to alter our notation. We now denote the variable $x_{ij}$ for $e=\{i,j\}\in E(K_n)$ as $x_e$. We will also write $\Delta_v$ for $v\in V(K_n)$.

\begin{theorem}
 The Nullstellensatz certificate given in Proposition \ref{prop:matchingcert} is a minimal degree certificate for $K_n$, $n$ odd.
\end{theorem}
\begin{proof}
By Proposition \ref{prop:matchingcert} we know that there exists a Nullstellensatz certificate of degree $\lfloor n/2\rfloor$. We work in the ring $R:= \C[x_{ij}: i\neq j \in [n]] / I$, where $I$ is generated by the second set of equations in Equations \ref{eq:firstmatch}. Let $\mathcal{M}$ be the set of matchings of $K_n$, and, for $M \in \mathcal{M}$ let $\mathbf{x}_M = \prod_{e \in M} x_e$. Since we are working in $R$, by Lemma \ref{lem:onlygoodmons}, we can write
\[
\Delta_v = \sum_{M \in \mathcal{M}} \alpha_{v,M} \mathbf{x}_M. 
\]
A Nullstellensatz certificate gives us that in $R$
\[
\sum_{v \in V(K_n)} \Delta_v \left( \sum_{e \in N(v)} x_e - 1 \right)=1.
\]
 The coefficient of $\mathbf{x}_M$ is given by
\[
\sum_{v \in V(K_n)} -\alpha_{v,M} + \sum_{e \in M} \sum_{v \in e} \alpha_{v,M \setminus e},
\]
which has to equal zero in a Nullstellensatz certificate if $|M|>0$. Now, if there is a Nullstellensatz certificate of degree $l  < \lfloor n/2\rfloor$, then, if $|M| =  l+1 $, we see that
\[
\mathcal{R}_M := \sum_{e = \{u,v\} \in M}{(\alpha_{u,M\setminus{e}}+\alpha_{v,M\setminus e})} = 0,
\]
and by edge transitivity of $G$, summing over these relations implies that
\[
\sum_{v \in V(K_n)} \sum_{\stackrel{M \in \mathcal{M}}{|M|= l}} \alpha_{v,M} =0.
\]
Furthermore, we have

$$0  = \sum_{v \in V(K_n)}\sum_{\stackrel{M \in \mathcal{M}}{|M| = l}}{-\alpha_{v,M}} + \sum_{e \in M}\sum_{v \in e}{\alpha_{v,M \setminus e}} \implies$$
 $$0 = \sum_{e \in M} \sum_{v \in e}{\alpha_{v,M\setminus e}}$$
 Then summing over the linear relations in the second line above gives
 \begin{align*}
 0& = (l-1)\sum_{v \in V(K_n)} \sum_{\stackrel{M \in \mathcal{M}}{|M| = l-1}}{\alpha_{v,M}}
\end{align*}
Repeating this, we obtain that 
\[
\sum_{v \in V} \alpha_{v, \emptyset} = 0,
\]
 which contradicts the assumption that the $\Delta_v$ give a Nullstellensatz certificate as we must have
 \[
\sum_{v \in V} \alpha_{v, \emptyset} = -1.
\]
 Thus we can conclude that there is no Nullstellensatz certificate where all $\Delta_v$ have degree at most $\lfloor n/2\rfloor-1$ in $R$. But we know from Proposition \ref{prop:matchingcert} that there exists a Nullstellensatz certificate where each $\Delta_v$ has degree $\lfloor n/2\rfloor$ and all other coefficient polynomials have degree at most $\lfloor n/2\rfloor$. So this Nullstellensatz certificate is of minimal degree.
\end{proof}

So we see that using NulLA to determine if a graph has a perfect matching using Equations \ref{eq:firstmatch} can be quite problematic. Since any graph with and odd number of vertices cannot have a perfect matching, the NulLA algorithm does a lot of work: for every $i\in[\lfloor n/2\rfloor]$, it determines if a system of linear equations in $\binom{a+i}{i}$ where $a=\binom{n}{2}$ variables, which is the number of monomials in the variables $x_{ij}$ of degree $i$. However, the NulLA algorithm could be made smarter by having it reject any graph on odd vertices before doing any linear algebra. This leads us to an open question:

\begin{question}
 Is there a family of graphs, each with even size, none of which have a perfect matching, such that the Nullstellensatz certificate given in Proposition \ref{prop:matchingcert} is of minimal degree?
\end{question}

We actually implemented the NulLA algorithm to try and find examples of graphs with high degree Nullstellensatz certificates for Equations \ref{eq:firstmatch}. The only ones were graphs containing odd cliques. This leads us to wonder if there are natural "bad graphs" for the degree of the Nullstellensatz certificate and if their presence as a subgraph determines the minimal degree. Formally:

\begin{question}
 Are there finitely many families of graphs $\mathcal{G}_1,\dots\mathcal{G}_k$ such that the degree of a minimal degree Nullstellensatz certificate for Equations \ref{eq:firstmatch} of a graph $G$ is determined by the largest subgraph of $G$ contained in one of the families $\mathcal{G}_i$?
\end{question}

\section{Conclusion}

In tackling decision problems, it is often a natural idea to rephrase them in some other area of mathematics and use algorithms from said area to see if performance can be improved. The NulLA algorithm is inspired by the idea of rewriting combinatorial decision problems as systems of polynomials and then using Gr\"obner basis algorithms from computational algebraic geometry to decide this problems quickly.

Amazingly, from a theoretical point of view, the rewriting of these problems as polynomial systems is not just a change of language. Lots of combinatorial data seems to come packaged with it. Throughout this paper, we have seen time and again that simply trying to solve the decision problem in graph theory might actually involve enumerating over many subgraphs. The theory of Nullstellensatz certificates is fascinating for the amount of extra information one gets for free by simply writing these problems as polynomial systems.

From an algorithmic viewpoint, our results suggest that one should be cautious about using the NulLA algorithm as a practical tool. The NulLA algorithm always finds a minimal degree certificate, and our theorems show that such certificates may entail solving a harder problem that the one intended. However, we do not know which minimal degree Nullstellensatz certificate will get chosen: maybe there are others that are less problematic algorithmically. 

Certainly, however, work should be done to understand which minimal degree Nullstellensatz certificates will actually be found by the algorithm if there is to be any hope in actual computational gains. We have analyzed the worst case scenario, but it is unclear how often it will arise in practice. 

\subsection*{Acknowledgments} We would like to thank Jeroen Zuiddam for coding a working copy of the NulLA algorithm for our use. 	The research leading to these results has received funding from the European Research Council under the European Union's Seventh Framework Programme (FP7/2007-2013) / ERC grant agreement No 339109.

\bibliographystyle{plain}
\bibliography{bibfile}
\end{document}